\documentclass[sn-mathphys-num]{sn-jnl}

\usepackage{graphicx}%
\usepackage{multirow}%
\usepackage{amsmath,amssymb,amsfonts}%
\usepackage{amsthm}%
\usepackage{mathrsfs}%
\usepackage[title]{appendix}%
\usepackage{xcolor}%
\usepackage{textcomp}%
\usepackage{manyfoot}%
\usepackage{booktabs}%
\usepackage{algorithm}%
\usepackage{algorithmicx}%
\usepackage{algpseudocode}%
\usepackage{listings}%

\theoremstyle{thmstyleone}%
\newtheorem{theorem}{Theorem}
\newtheorem{proposition}[theorem]{Proposition}%
\newtheorem{corollary}[theorem]{Corollary}
\newtheorem{lemma}[theorem]{Lemma}
\newtheorem{Remark}[theorem]{Remark}
\theoremstyle{thmstyletwo}%
\newtheorem{example}{Example}%

\theoremstyle{thmstylethree}%
\newtheorem{definition}{Definition}%

\begin{document}

\title[Excess of Fusion Frames: A Comprehensive Approach]{Excess of Fusion Frames: A Comprehensive Approach}


\author[1]{\fnm{Ehsan} \sur{Ameli}}\email{eh.ameli@hsu.ac.ir}

\author*[1]{\fnm{Ali Akbar} \sur{Arefijamaal}}\email{arefijamaal@hsu.ac.ir}
\equalcont{These authors contributed equally to this work.}

\author[2,3]{\fnm{Fahimeh} \sur{Arabyani Neyshaburi}}\email{fahimeh.arabyani@gmail.com}
\equalcont{These authors contributed equally to this work.}

\affil[1]{\orgdiv{Department of Mathematics and Computer Sciences}, \orgname{Hakim Sabzevari University}, \state{Sabzevar}, \country{Iran}}

\affil[2]{\orgdiv{Department of Mathematical Sciences}, \orgname{Ferdowsi University of Mashhad}, \city{Mashhad}, \country{Iran}}

\affil[3]{\orgdiv{Department of Mathematics}, \orgname{University of Neyshabur}, \city{Neyshabur}, \postcode{P.O.Box 91136-899}, \country{Iran}}


\abstract{Computing the excess as a method of measuring the redundancy of frames was recently introduced to address certain issues in frame theory. In this paper, the concept of excess for fusion frames is studied. Then, several explicit methods are provided to compute the excess of fusion frames and their $Q$-duals. In particular, some upper bounds for the excess of $Q$-dual fusion frames are established. It turns out that, unlike ordinary frames, for every $n \in \Bbb{N}$ we can provide a fusion frame with its $Q$-dual whose the difference of their excess is $n$. Furthermore, the connection between the excess of fusion frames and their orthogonal complement is completely characterized. Finally, several examples are exhibited to confirm the obtained results.}




\maketitle

\section{Introduction}\label{sec1}

The excess carries out a crucial role in frame theory and signal processing, as it allows for a more robust representation of signals in the presence of noise and common types of interference. The excess is a rather crude way of measuring the redundancy of frames. In fact, it is defined as the greatest number of elements that can be removed from a frame in a Hilbert space and still leave a set with the same closed linear span. Recently, accurate methods have been provided for computing the excess of frames and g-frames \cite{Deficit, Besselian, Nga}. 

Fusion frame theory is a generalization of ordinary frames in separable Hilbert spaces, introduced by Casazza and Kutyniok in \cite{frame of subspace}. Fusion frames play an important role in many applications in mathematical analysis and engineering, including coding theory, filter bank theory, signal and image processing and wireless communications and many other fields \cite{existence, frame of subspace, Excess 1, Nga, Rahimi}. In this regard, understanding the redundancy of fusion frames is a fundamental issue that has an impact on many applications. Many concepts of frame theory have been generalized to the fusion frame setting \cite{Arabyani dual, A.A.SH, Nga,  Rahimi, mitra sh.}. In most of those results, the calculation of the excess and the construction of Riesz bases are of key importance. In this respect, having a notion of the excess of fusion frames allows us to generate fusion Riesz bases. Motivated by this point of view and the recent work of Balan et al. \cite{Deficit} in the direction of excess of frames in Hilbert spaces, we propose an approach to determine the excess of fusion frames. We build a connection between the excess of fusion frames and their local frames, and subsequently investigate the differences between these two. This approach would be interesting to find out whether the excess of fusion frames can be defined as the removal of redundant subspaces. In \cite{Rahimi}, some approaches are presented to determine the redundancy of fusion frames. Moreover, the redundancy function and the concept of upper and lower redundancies are introduced in \cite{quantitative} as a quantitative measure for computing redundancy in ordinary frames, which is not aligned with our intuitive understanding of the excess in Hilbert spaces, as we will discuss in this paper. Hence, this survey focuses exclusively on the study of the excess of fusion frames.\color{black}

Our next objective is to provide an explicit method for computing the excess of fusion frames. Furthermore, we try to obtain the excess of $Q$-dual fusion frames which are introduced in \cite{Heineken}. One advantage of these dual fusion frames with respect to alternate (G\~{a}vru\c{t}a) dual fusion frames \cite{Gavruta} is that they can be readily obtained from the left inverses of the analysis operator of fusion frames. So, the following questions naturally arise: What is the relationship between the excess of a fusion frame and its $Q$-dual fusion frame? For a given fusion frame, is there a $Q$-dual fusion frame with the same excess? How can we establish the connection between the excess of a fusion frame and its orthogonal complement fusion frame? In this study, we provide several responses to all of these questions.

The remainder of the paper is structured as follows. In Section 2, we briefly review frames and fusion frames in Hilbert spaces. In Section 3, we take an approach to introduce the excess of fusion frames and thereby we make a connection with the excess of local frames. We then present an explicit method to compute the excess of fusion frames. Furthermore, we study the relationship between the excess of fusion frames and their orthogonal complement fusion frames and we compute the excess of a class of fusion frames. Finally, Section 4 is devoted to study the excess of $Q$-dual fusion frames and to present some methods for computing their excess. Moreover, we give several examples of $Q$-dual fusion frames and we compute their excess by applying the obtained results. Specifically, we show that, unlike ordinary frames, there are $Q$-dual fusion frames in which their excess are much more than the excess of own fusion frame.

\section{Preliminaries and Notations} 
We review the basic definitions and primary results of frames and fusion frames. Throughout this paper, we suppose that $\mathcal {H}$ is a separable Hilbert space and $\mathcal {H}_n$ an n-dimensional Hilbert space. Also, $I$ and $J$ denote countable index sets and $I_{\mathcal {H}}$ denotes the identity operator on $\mathcal {H}$. We denote the set of all bounded operators on $\mathcal {H}$ by $B(\mathcal {H})$. Moreover, we denote the range and null space of $T \in B(\mathcal {H})$ by $R(T)$ and $N(T)$, respectively. Finally, the orthogonal projection onto a closed subspace $W$ of $\mathcal {H}$ is denoted by $ \pi_{W} $.

Recall that a \textit{frame} for $\mathcal {H}$ is a sequence $\{f_{i}\}_{i \in I}$ of vectors in $\mathcal {H}$ such that there are constants $0<A\leq B<\infty$ satisfying $A \|f\|^{2}\leq \sum _{i \in I} |\langle f,f_{i}\rangle|^{2} \leq B\|f \|^{2},$ for all $f \in \mathcal{H}$. The constants $A$ and $B$ are called \textit{frame bounds}.
Those sequences which satisfy only the upper inequality are called \textit{Bessel sequences}. If $\{f_{i}\}_{i \in I}$ is a Bessel sequence, the \textit{synthesis operator} $T: \ell^{2}(I) \rightarrow \mathcal{H}$ is defined by $ T(\{c_{i}\}_{i \in I})= \sum_{i \in I} c_{i} f_{i}$. 
The adjoint operator $T^{*}:\mathcal{H} \rightarrow \ell^{2}(I)$, so-called \textit{analysis operator} is given by $T^{*}f =\{\langle f , f_{i}\rangle\}_{i \in I}$. Moreover, the \textit{frame operator} of $\{f_{i}\}_{i \in I}$, is defined by 
$Sf=TT^{*}f=\sum _{i \in I}\langle f,f_{i}\rangle f_{i}$, for all $f \in \mathcal{H}$. That is a positive, self-adjoint as well as invertible operator \cite{Chr08} provided that $\{f_{i}\}_{i \in I}$ is a frame. Hence, we obtain
\begin{equation*}
f=S^{-1}Sf=\sum _{i \in I}\left\langle  f,S^{-1}f_{i}\right\rangle  f_{i}, \quad(f \in \mathcal{H}).
\end{equation*}
A Bessel sequence $\{g_{i}\}_{i \in I}$ is called an \textit{alternate dual} of $\{f_{i}\}_{i \in I}$ if $\sum_{i \in I}\langle f,f_{i}\rangle g_{i}=f,$ for all $f \in \mathcal{H}$. It is shown that for every dual frame $\{g_{i}\}_{i \in I},$ there exists a Bessel sequence $\{u_{i}\}_{i \in I}$ such that \cite{Zekaee} $g_{i}=S^{-1}f_{i}+u_{i}~(i \in I),$ where $\sum_{i \in I}\langle f,f_{i}\rangle u_{i}=0$ for all $f \in \mathcal{H}$.

Frames which are not bases are overcomplete, i.e. there exist proper subsets of the frame which are frame as well. The \textit{excess} of the frame is the greatest integer n such that n elements can be removed from the frame and still have a complete set, or $\infty$ if there is no upper bound to the number of elements that can be omitted. By \cite[Lemma 4.1]{Deficit}, the excess of a frame $\mathcal {F}$ is connected to the dimension of the kernel of $T_{\mathcal {F}}$. In fact, $e(\mathcal {F})=\textnormal{dim}N(T_{\mathcal {F}})$. Also, dual frames have the same excess \cite{Bakic}. The following proposition provides a method to compute the excess of frames.
\begin{proposition}\cite{Deficit}\label{balan deficit}
Let $\mathcal {F}=\{f_{i}\}_{i \in I}$ be a frame in a Hilbert space $\mathcal {H}$ with the canonical dual frame $\mathcal {\widetilde{F}}=\{\tilde{f_{i}}\}_{i \in I}$. Then the excess of $\mathcal {F}$ is
\begin{equation*}
e(\mathcal {F})=\sum_{i \in I}\big(1-\big\langle f_{i},\tilde{f_{i}}\big\rangle\big).
\end{equation*}
\end{proposition}

For more details on the frame theory we refer the reader to \cite{Cas00, Chr08, Duffin, Besselian}. During the last decade, fusion frame theory has been a growing area of research that plays an important role in many applications. Consider a family of closed subspaces $\{W_{i}\}_{i \in I}$ of $\mathcal {H}$ and $ \{\omega_{i}\}_{i \in I} $ as a family of weights, i.e.  $\omega_{i}>0 $, $(i \in I)$. Then $ \textit{W}=\{(W_{i},\omega_{i})\}_{i \in I} $ is called a \textit{fusion frame} \cite{frame of subspace} for $\mathcal {H}$ if there exist constants $0<A\leq B<\infty$ such that 
\begin{equation}\label{fusion frame def}
A \Vert f\Vert^{2}\leq \sum _{i \in I} \omega_{i}^2\Vert \pi_{W_{i}}f \Vert ^{2} \leq B\Vert f \Vert ^{2}, \quad(f \in \mathcal{H}).
\end{equation}
The constants $A$ and $B$ are called the \textit{fusion frame bounds}. If we only have the upper bound in \eqref{fusion frame def}, we call $ W $ is a \textit{fusion Bessel sequence}. The family $\{W_{i}\}_{i \in I}$ is called \textit{A-tight} fusion frame if $ A=B $, and \textit{Parseval} if $ A=B=1 $. If $ \omega_{i}=\omega $ for all $ i\in I $, then $ W $ is called \textit{$ \omega $-uniform} fusion frame and if $\textnormal{dim}W_{i}=n$ for all $ i\in I$, then $W$ is called \textit{n-equi-dimensional} fusion frame. A family of closed subspaces $\{W_{i}\}_{i \in I} $ is said to be a \textit{fusion orthonormal basis} when $\mathcal {H}$ is the orthogonal sum of the subspaces $ W_{i}$ and it is a \textit{Riesz decomposition} of $\mathcal {H}$, if for every $ f \in \mathcal{H} $ there is a unique choice of $ f_{i}\in W_{i} $ such that $ f=\sum _{i \in I} f_{i} $. A fusion frame is said to be \textit{exact}, if it ceases to be a fusion frame whenever anyone of its element is deleted. A family of closed subspaces $\{W_{i}\}_{i \in I} $ is called a \textit{fusion Riesz basis} whenever it is complete for $\mathcal {H}$ and there exist positive constants $C$ and $D$ such that for every finite subset $ J\subset I $ and arbitrary vector $ f_{j}\in W_{j}~ (j \in J)$, we have
\begin{equation*}\label{fusion Riesz basis}
C \sum _{j \in J}\Vert f_{j} \Vert ^{2} \leq \bigg\Vert \sum _{j \in J} \omega_{j} f_{j}\bigg\Vert ^2 \leq D\sum _{j \in J}\Vert f_{j} \Vert ^{2}.
\end{equation*}
Recall that for each sequence $\{W_{i}\}_{i \in I}$ of closed subspaces in $\mathcal {H}$, the space
\begin{equation*}
\left( \sum_{i \in I}\bigoplus W_{i}\right)_{\ell ^2}=\left\lbrace \{f_{i}\}_{i \in I}: f_{i}\in W_{i}~,~ \sum_{i \in I}\Vert f_{i} \Vert ^{2}< \infty \right\rbrace , 
\end{equation*}
with the inner product
\begin{equation*}
\big\langle  \{f_{i}\}_{i \in I} , \{g_{i}\}_{i \in I} \big\rangle  = \sum_{i \in I}\langle f_{i} , g_{i}\rangle ,
\end{equation*}
constitutes a Hilbert space. Henceforth, for the sake of brevity, we write $\bigoplus_{i \in I} W_{i}$ instead of $\left( \sum_{i \in I}\bigoplus W_{i}\right)_{\ell ^2}$. For a Bessel sequence $W=\{(W_{i},\omega_{i})\}_{i \in I}$, the \textit{synthesis operator} \hbox{$ T_{W}: \bigoplus_{i \in I} W_{i} \rightarrow \mathcal {H}$} is defined by 
\begin{equation*}
T_{W}(\{f_{i}\}_{i \in I})=\sum_{i \in I}\omega_{i}f_{i}, \quad \{f_{i}\}_{i \in I} \in \bigoplus_{i \in I} W_{i}.
\end{equation*}
Its adjoint operator $ T^{*}_{W}: \mathcal {H} \rightarrow \bigoplus_{i \in I} W_{i}$, which is called the \textit{analysis operator}, is given by
$ T^{*}_{W}f=\{\omega_{i}\pi_{W_{i}}f\}_{i \in I}$ for all $f \in \mathcal {H}$. If $W=\{(W_{i},\omega_{i})\}_{i \in I}$ is a fusion frame, the \textit{fusion frame operator} $ S_{W}:\mathcal {H} \rightarrow \mathcal {H}$ defined by
\begin{equation*}
S_{W}f=T_{W}T^{*}_{W}f=\sum_{i \in I}\omega_{i}^{2}\pi_{W_{i}}f,
\end{equation*}
is positive, self-adjoint as well as invertible. Thus we have the following reconstruction formula \cite{frame of subspace}:
\begin{equation*}
f=\sum_{i \in I}\omega_{i}^{2}S_{W}^{-1}\pi_{W_{i}}f, \quad(f \in \mathcal {H}).
\end{equation*}
In \cite{frame of subspace}, it has been proved that $W=\{(W_{i},\omega_{i})\}_{i \in I} $ is a Parseval fusion frame if and only if $S_{W}=I_{\mathcal {H}}$. The family $\widetilde{W}:=\left\lbrace (S_{W}^{-1}W_{i},\omega_{i})\right\rbrace _{i \in I}$, which is also a fusion frame, is called the \textit{canonical dual} of $W$. Generally, a Bessel sequence  $\{(V_{i},\upsilon_{i})\}_{i \in I}$ is called an \textit{alternate (G\~{a}vru\c{t}a) dual} of $W$, whenever
\begin{equation*}
 f=\sum_{i \in I}\omega_{i}\upsilon_{i}\pi_{V_{i}}S_{W}^{-1}\pi_{W_{i}}f,\quad(f \in \mathcal {H}).
\end{equation*}
Let $W=\{(W_{i},\omega_{i})\}_{i \in I} $ be a fusion frame for $\mathcal {H}$. A fusion Bessel sequence $V=\{(V_{i},\upsilon_{i})\}_{i \in I}$ is a dual of $W$ if and only if \cite{Osgooei}
\begin{equation*}
T_{V}\varphi_{VW} T_{W}^{*}=I_{\mathcal {H}},
\end{equation*}
where the bounded operator $\varphi_{VW}:\bigoplus_{i \in I} W_{i}\rightarrow \bigoplus_{i \in I} V_{i} $ is given by
\begin{equation*}
\varphi_{VW}\big(\{f_{i}\}_{i \in I}\big)=\left\lbrace \pi_{V_{i}}S_{W}^{-1}f_{i}\right\rbrace _{i \in I} .
\end{equation*}

We conclude this section by presenting some results that will be useful in the subsequent sections of this paper.
\begin{theorem}\label{Local theorem}
\cite{frame of subspace} Let $\{W_{i}\}_{i \in I}$ be a family of closed subspaces of $\mathcal {H}$, $ \omega _{i} >0 $ and $\{f_{i,j}\}_{j \in J_{i}}$ be a frame (Riesz basis) for $W_{i}$ with frame bounds $A_{i}$ and $B_{i}$ such that
\begin{equation*}
0 < A=\textnormal{inf}_{i \in I} A_{i}\leq \textnormal{sup}_{i \in I} B_{i}=B < \infty .
\end{equation*}
Then the following conditions are equivalent:
\begin{itemize}
\item[(i)] $W=\{(W_{i},\omega_{i})\}_{i \in I}$ is a fusion frame (fusion Riesz basis) for $\mathcal {H}$.
\item[(ii)] $ \mathcal{F}=\{\omega_{i}f_{i,j}\}_{i \in I, j \in J} $ is a frame (Riesz basis) for $\mathcal {H}$.
\end{itemize}
\end{theorem}
The frame $\mathcal{F}$ in the above theorem is called the \textit{local frame} of $W$.

\begin{proposition}\cite{frame of subspace, mitra sh.} \label{mitra sh.} 
Let $W=\{(W_{i},\omega_{i})\}_{i \in I} $ be a fusion frame for $\mathcal {H}$. Then the following are equivalent:
\begin{itemize}
\item[(i)] W is a fusion Riesz basis.
\item[(ii)] $S_{W}^{-1}W_{i}\perp W_{j} ~ for~all~ i , j \in I,~ i \neq j $.
\item[(iii)] $ \omega_{i}^{2}\pi_{W_{i}}S_{W}^{-1}\pi_{W_{j}}=\delta_{i,j}\pi_{W_{j}}~ for~all~ i , j \in I $.
\end{itemize}
\end{proposition}
Throughout this note, if $\{(W_{i},\omega_{i})\}_{i \in I} $ is a fusion frame for $\mathcal {H}_{n},$ then it is traditionally assumed that $|I|<\infty$.
\section{Excess of Fusion Frames}
The excess is a way of measuring the redundancy of fusion frames, see also \cite{Excess 1}. First, we recall the concept of excess for fusion frames.
\begin{definition}\cite{Excess 1}
Let $W=\{(W_{i},\omega_{i})\}_{i \in I} $ be a fusion frame for $\mathcal {H}$ with the synthesis operator $T_{W}$. The \textit{excess} of $W$ is defined as 
\begin{equation*}
e(W)=\textnormal{dim}N(T_{W}).
\end{equation*}
\end{definition}

Suppose that $\{e_{i}\}_{i=1}^n $ is an orthonormal basis for $\mathcal {H}_{n}$, $W_{i}=\textnormal{span} \{e_{i}\}$ and $\omega_{i}=1$, $i=1,\ldots,n$. Take 
\begin{equation*}
W=\left\lbrace W_{1},W_{1},W_{2},W_{2},\ldots, W_{n},W_{n} \right\rbrace .
\end{equation*}
Obviously $\textnormal{dim}N(T_{W})=n$, while $W$ is a 2-tight fusion frame with the uniform redundancy 2, see \cite{Rahimi} for more details. However, our approach provides a precise description of which part of each subspace can be considered as redundancy.  The present paper, is concerned exclusively with the excess of fusion frames. 
\subsection{Computational view of point}
As we observed, in ordinary frames the excess of a frame is defined as the greatest number of elements that can be removed and yet leave a set with the same closed linear span. Since the subspaces of a fusion frame are not disjoint, in general, we cannot provide an analogous approach to the excess of fusion frames. However, we can connect it to the excess of local frames. In the sequel, we compute the excess of fusion frames using their local frames. We do this first through local frames obtained from Riesz bases. 
\begin{proposition}\label{excess}
Let $\mathcal{W}=\{(W_{i},\omega_{i})\}_{i \in I} $ be a fusion frame for $\mathcal {H}$ and $\mathcal{F}=\{\omega_{i}f_{i,j}\}_{i \in I, j \in J_{i}}$ its local, where $\{f_{i,j}\}_{j \in J_{i}}$ is a Riesz basis for $W_{i},$ for all $i \in I$. Then 
\begin{equation*}
e(\mathcal{W})=e(\mathcal{F}).
\end{equation*}
\end{proposition}

\begin{proof}
Suppose that $\{g_{i}\}_{i\in I}\in \bigoplus_{i \in I} W_{i}$, then $g_{i}=\sum_{j\in J_{i}}c_{i,j}f_{i,j}$ ($i\in I$), where $\{c_{i,j}\}_{j\in J_{i}} \in l^{2}(J_{i})$. If $T_{W}\big(\{g_{i}\}_{i\in I}\big)=0$, then $\sum_{i\in I, j\in J_{i}}c_{i,j}\omega_{i}f_{i,j}=\sum_{i\in I}\omega_{i}g_{i}=0$, which implies that $T_{\mathcal{F}}\big(\{c_{i,j}\}_{i\in I, j\in J_{i}}\big)=0$. Hence, there exists a bijective correspondence between $N(T_{W})$ and $N(T_{\mathcal{F}})$. Therefore, $e(W)=e(\mathcal{F})$.
\end{proof}
This statement leads directly to the conclusion that the excess of $W$ is not affected by the weights. Indeed, if $\mathcal{F}$ is the local frame introduced in the above proposition, then it is sufficient to show that $e(\mathcal{F})=e(\mathcal{F}^{1}),$ where $\mathcal{F}^{1}=\{f_{i,j}\}_{i \in I, j \in J_{i}}$ is a local frame of $W^{1}=\{(W_{i},1)\}_{i \in I}.$ To this end, suppose that $e(\mathcal{F}^{1})<\infty$. Then there exist $L \subseteq I$ and $K_{i} \subseteq J_{i},~(i \in L)$ such that $\{f_{i,j}\}_{i \in L, j \in K_{i}}$ is a Riesz basis and
\begin{equation*}
\sum_{i \in L}|J_{i}\smallsetminus K_{i}|+\sum_{i \in I\smallsetminus L}|J_{i}|=e(\mathcal{F}^{1})<\infty .
\end{equation*}
Obviously, $\{\omega_{i}f_{i,j}\}_{i \in L, j \in K_{i}}$ is also a Riesz basis. In particular, $e(\mathcal{F})=e(\mathcal{F}^{1}).$ Now, if $e(\mathcal{F}^{1})=\infty,$ we claim that $e(\mathcal{F})=\infty.$ Otherwise, assume that $e(\mathcal{F})<\infty.$ Repeating the above argument, it follows that $e(\mathcal{F}^{1})<\infty,$ which is a contradiction. Hence, $W$ and $W^{1}$ have the same excess, see also \cite[Theorem 6.8]{Excess 1} for another proof. In spite of this fact, we take weights into account in our examples. The following example confirms Proposition \ref{excess}.

\begin{example}\label{intersection}
Consider $W_{1}=\Bbb R^{2}\times\{0\}$ and $W_{2}=\{0\}\times \Bbb R^{2}$. Then \hbox{$W=\{(W_{i},\omega_{i})\}_{i=1}^2$} is a fusion frame for $\mathcal {H}=\Bbb R^{3}$ with the local frame
\begin{equation*}
\mathcal{F}=\big\{\omega_{1}(1,1,0),\omega_{1}(1,-1,0),\omega_{2}(0,1,0),\omega_{2}(0,-1,1)\big\}.
\end{equation*}
It is easily seen that $e(W)=e(\mathcal{F})=1$.
\end{example}
In this example, it is worth noting that $e(W)=1$ does not mean that a subspace can be removed, it actually means that a certain vector in a local frame obtained from Riesz bases can be removed. Let $W=\{(W_{i},\omega_{i})\}_{i \in I}$ be a fusion frame for $\mathcal {H}_n$. One direct consequence of the rank-nullity theorem is that 
\begin{equation}\label{finite dimensional}
e(W)=\sum_{i\in I}\textnormal{dim} W_{i}-\textnormal{dim} \mathcal {H}_n.
\end{equation}
The following theorem provides a method to compute the excess of fusion frames in infinite dimensional Hilbert spaces, see also \cite{Nga} for g-frames.
\begin{theorem}\label{th}
Let $W=\{(W_{i},\omega_{i})\}_{i \in I}$ be a fusion frame for $\mathcal {H}$. Then 
\begin{equation*}
e(W)=\sum_{i\in I,j\in J_{i}}\left( 1-\omega_{i}^{2}\left\langle  e_{i,j},S_{W}^{-1}e_{i,j} \right\rangle  \right) .
\end{equation*}
Specially, if $\textnormal{dim} W_{i}<\infty$ for all $i \in I$, then
\begin{equation*}
e(W)=\sum_{i\in I}\left( \textnormal{dim} W_{i}-\omega_{i}^{2}\textnormal{trace}\left(\pi_{W_{i}}S_{W}^{-1}\pi_{W_{i}}\right)\right).
\end{equation*}
\end{theorem}
\begin{proof}
Let $\{e_{i,j}\}_{j \in J_{i}}$ be an orthonormal basis for $W_{i}~(i\in I)$ and denote elements $E_{i,j}$ of $\bigoplus_{i \in I} W_{i}$ by
\begin{equation}\label{oonnbb}
(E_{i,j})_{k}=\begin{cases}
e_{i,j},~& i=k
\\
0,~& i\neq k.
\end{cases}
\end{equation}
It is known that $\{E_{i,j}\}_{i\in I, j \in J_{i}}$ is an orthonormal basis for $\bigoplus_{i \in I} W_{i}$. The \hbox{orthogonal} projection $P$ of $\bigoplus_{i \in I}W_{i}$ onto $N(T_{W})$ is given by $P=I_{\bigoplus W_{i}}-T_{W}^{*}S_{W}^{-1}T_{W}.$
So, we obtain
\begin{align*}
e(W)=\textnormal{dim}N(T_{W})&=\textnormal{trace}(P)
\\&=\sum_{i\in I,j\in J_{i}}\big\langle E_{i,j},PE_{i,j} \big\rangle
\\&=\sum_{i\in I,j\in J_{i}}\left( 1-\left\langle T_{W}E_{i,j},S_{W}^{-1}T_{W}E_{i,j} \right\rangle \right) 
\\&=\sum_{i\in I,j\in J_{i}}\left( 1-\omega_{i}^{2}\left\langle  e_{i,j},S_{W}^{-1}e_{i,j} \right\rangle  \right).
\end{align*}
In particular, if $\textnormal{dim} W_{i}<\infty$ for all $i \in I$, then the above computations yield
\begin{align*}
e(W)&=\sum_{i\in I,j\in J_{i}}\left( 1-\omega_{i}^{2}\left\langle  e_{i,j},S_{W}^{-1}e_{i,j} \right\rangle  \right) 
\\&=\sum_{i\in I}\left( \textnormal{dim} W_{i}-\omega_{i}^{2}\sum_{j\in J_{i}}\left\langle  e_{i,j},\pi_{W_{i}}S_{W}^{-1}\pi_{W_{i}}e_{i,j} \right\rangle  \right) 
\\&=\sum_{i\in I}\left( \textnormal{dim} W_{i}-\omega_{i}^{2}\textnormal{trace}\left( \pi_{W_{i}}S_{W}^{-1}\pi_{W_{i}}\right) \right) .
\end{align*}
\end{proof}
\begin{example}\label{excess of fusion frames}
\begin{itemize}
\item[(1)] Consider $n \in \Bbb N$, $W_{1}=\Bbb R^{n}\times\{0\}$ and $W_{2}=\{0\}\times \Bbb R^{n}.$ Then $W=\{(W_{i},\omega_{i})\}_{i=1}^2$ is a fusion frame for $\mathcal {H}=\Bbb R^{n+1} $. Moreover, $e(W)=\textnormal{dim}N(T_{W})=n-1$. Now, according to \eqref{finite dimensional}, we obtain
\begin{align*}
e(W)&=\sum_{i=1}^2\textnormal{dim} W_{i}-\textnormal{dim} \mathcal {H}
\\&=2n-(n+1)=n-1.
\end{align*}
\item[(2)] Let $\{e_{i}\}_{i=1}^\infty$ be an orthonormal basis for $\mathcal {H}$. Consider $W_{i}=\textnormal{span}\{e_{i},e_{i+1}\}$, for all $i \in \Bbb{N}$. Then $W=\{(W_{i},\omega)\}_{i=1}^\infty$ is a 2-equi-dimensional fusion frame for $\mathcal {H}$ with the fusion frame operator $S_{W}=\textnormal{diag}\left( \omega ^{2},2\omega ^{2},2\omega ^{2},\ldots \right) .$ Direct computations show that $S_{W}^{-1}f=\omega^{-2}\left( f-\sum_{i=2}^\infty \langle f,e_{i} \rangle \frac{e_{i}}{2}\right),$ for all $f \in \mathcal {H}$. Therefore, by Theorem \ref{th} we obtain
\begin{align*}
e(W)&=\sum_{i=1}^{\infty}\left( \textnormal{dim} W_{i}-\omega ^{2}\textnormal{trace}\left( \pi_{W_{i}}S_{W}^{-1}\pi_{W_{i}}\right) \right) 
\\&=\left[ 2-\omega ^{2}\left( \omega ^{-2}+\frac{\omega ^{-2}}{2}\right) \right] +\sum_{i=2}^{\infty}\left( 2-\omega ^{2}\omega ^{-2}\right) =\infty.
\end{align*}

\item[(3)] Let $\{e_{i}\}_{i=1}^\infty$ be an orthonormal basis for $\mathcal {H}$ and $n \in \Bbb N$. Consider
\begin{equation*}
W_{i}:=\begin{cases}
\textnormal{span}\{e_{i},e_{i+1}\},~& 1 \leq i \leq n,
\\
\textnormal{span}\{e_{i+1}\},~& i > n.
\end{cases}
\end{equation*}
Then $W=\{(W_{i},\omega)\}_{i=1}^\infty$ is a fusion frame for $\mathcal {H}$ and
\begin{equation*}
S_{W}^{-1}f=\omega^{-2}\left( f-\sum_{i=2}^{n} \langle f,e_{i} \rangle \frac{e_{i}}{2}\right),\quad(f \in \mathcal {H}).
\end{equation*}
Moreover,
\begin{equation*}
\sum_{i=n+1}^{\infty}\left( \textnormal{dim} W_{i}-\omega ^{2}\textnormal{trace}\left( \pi_{W_{i}}S_{W}^{-1}\pi_{W_{i}}\right) \right) =\sum_{i=n+1}^{\infty}\left( 1-\omega ^{2}\omega ^{-2}\right) =0.
\end{equation*}
Hence, we obtain
\begin{align*}
e(W)&=\sum_{i=1}^{\infty}\left( \textnormal{dim} W_{i}-\omega_{i}^{2}\textnormal{trace}\left( \pi_{W_{i}}S_{W}^{-1}\pi_{W_{i}}\right) \right) 
\\&=\sum_{i=1}^{n}\left( \textnormal{dim} W_{i}-\omega_{i}^{2}\textnormal{trace}\left( \pi_{W_{i}}S_{W}^{-1}\pi_{W_{i}}\right) \right) 
\\&=2\left[ 2-\omega ^{2}\left( \omega ^{-2}+\frac{\omega ^{-2}}{2}\right) \right] +\sum_{i=2}^{n-1}\left( 2-\omega ^{2}\omega ^{-2}\right) =n-1.
\end{align*}

\item[(4)] Let $\{e_{i}\}_{i \in \Bbb Z}$ be an orthonormal basis for $\mathcal {H}$, $W_{1}=\overline{\textnormal{span}}_{i\geq 0}\{e_{i}\}$ and $W_{2}=\overline{\textnormal{span}}_{i\leq 0}\{e_{i}\}$. In \cite{frame of subspace}, it has been shown that $W=\{(W_{i},\omega)\}_{i=1}^2$ is an exact fusion frame which is not Riesz basis. A straightforward calculation shows that $S_{W}^{-1}f=\omega^{-2}\left( f-\langle f,e_{0} \rangle \frac{e_{0}}{2}\right) $, for all $f \in \mathcal {H}$. Therefore, it follows from Theorem \ref{th} that
\begin{align*}
e(W)&=\sum_{i\in \Bbb Z}\left( 1-\omega^{2}\left\langle  e_{i},S_{W}^{-1}e_{i} \right\rangle  \right) +\left( 1-\omega^{2}\left\langle  e_{0},S_{W}^{-1}e_{0} \right\rangle \right) 
\\&=\sum_{i\in \Bbb Z}\left( 1-\left\langle  e_{i},\left( e_{i}-\langle e_{i},e_{0} \rangle\frac{e_{0}}{2}\right)  \right\rangle \right) +\frac{1}{2}
\\&=\sum_{i\in \Bbb Z}\frac{\vert\langle e_{i},e_{0}\rangle \vert^{2}}{2}+\frac{1}{2}=1.
\end{align*}
\end{itemize}
\end{example}
Suppose that $W=\{(W_{i},\omega_{i})\}_{i \in I}$ is a fusion frame for $\mathcal {H}$ and $\{e_{j}\}_{j \in J}$ is an orthonormal basis for $\mathcal {H}$. Then $\mathcal{F}_{W}=\{\omega_{i}\pi_{W_{i}}e_{j}\}_{i \in I, j \in J}$ is a local frame of $W$ \cite{A.A.SH}. In the next theorem, we investigate the relationship between the excess of $W$ and $\mathcal{F}_{W}$.
\begin{theorem}\label{excess new local}
Let $W=\{(W_{i},\omega_{i})\}_{i \in I}$ be a fusion frame for $\mathcal{H}$ such that \hbox{$\textnormal{dim}W_{i}<\infty$} for all $i \in I$. Let $\mathcal{F}_{W}=\{\omega_{i}\pi_{W_{i}}e_{j}\}_{i \in I, j \in J}$. Then
\begin{equation*}
e(\mathcal{F}_{W})+\sum_{i\in I}\textnormal{dim}W_{i}=|I|\textnormal{dim}\mathcal{H}+e(W).
\end{equation*}
In particular, if $\textnormal{dim}\mathcal{H}<\infty,$ then $e(\mathcal{F}_{W})=(|I|-1)\textnormal{dim}\mathcal{H}$.
\end{theorem}
\begin{proof}
It is easily observed that the frame operators $S_{W}$ and $S_{\mathcal{F}_{W}}$ are the same. Indeed, for each $f \in \mathcal{H}$ we have
\begin{align*}
S_{\mathcal{F}_{W}}f&=\sum_{i\in I}\sum_{j\in I}\langle f,\omega_{i}\pi_{W_{i}}e_{j} \rangle \omega_{i}\pi_{W_{i}}e_{j}
\\&=\sum_{i\in I}\sum_{j\in I}\omega_{i}^{2}\pi_{W_{i}}\langle \pi_{W_{i}}f,e_{j} \rangle e_{j}
\\&=\sum_{i\in I}\omega_{i}^{2}\pi_{W_{i}}f=S_{W}f.
\end{align*}
Applying Proposition \ref{balan deficit} yields
\begin{align*}
e(\mathcal{F}_{W})&=\sum_{i\in I}\sum_{j\in I}\left( 1-\left\langle  \omega_{i}\pi_{W_{i}}e_{j},\omega_{i}S_{W}^{-1}\pi_{W_{i}}e_{j} \right\rangle \right) 
\\&=\sum_{i\in I}\bigg( \textnormal{dim}\mathcal{H}-\omega_{i}^{2}\sum_{j\in I}\left\langle  e_{j},\pi_{W_{i}}S_{W}^{-1}\pi_{W_{i}}e_{j} \right\rangle \bigg)
\\&=\sum_{i\in I}\left( \textnormal{dim}\mathcal{H}-\omega_{i}^{2}\textnormal{trace}\left( \pi_{W_{i}}S_{W}^{-1}\pi_{W_{i}}\right) \right) .
\end{align*}
Therefore, by using Theorem \ref{th} we obtain
\begin{align*}
e(\mathcal{F}_{W})+\sum_{i\in I}\textnormal{dim}W_{i}&=\sum_{i\in I}\left( \textnormal{dim}\mathcal{H}-\omega_{i}^{2}\textnormal{trace}\left( \pi_{W_{i}}S_{W}^{-1}\pi_{W_{i}}\right) +\textnormal{dim}W_{i}\right) 
\\&=|I|\textnormal{dim}\mathcal{H}+\sum_{i\in I}\left( \textnormal{dim} W_{i}-\omega_{i}^{2}\textnormal{trace}\left( \pi_{W_{i}}S_{W}^{-1}\pi_{W_{i}}\right) \right) 
\\&=|I|\textnormal{dim}\mathcal{H}+e(W).
\end{align*}
Moreover, if $\textnormal{dim}\mathcal{H}<\infty,$ then it gives
\begin{equation*}
e(\mathcal{F}_{W})=|I|\textnormal{dim}\mathcal{H}-\sum_{i\in I}\textnormal{dim}W_{i}+e(W)=(|I|-1)\textnormal{dim}\mathcal{H}.
\end{equation*}
\end{proof}
Consider the fusion frame $W=\{(W_{i},\omega_{i})\}_{i=1}^{2}$ introduced in Example \ref{excess of fusion frames}(1) with the local frame $\mathcal{F}_{W}=\{\omega_{i}\pi_{W_{i}}e_{j}\}_{i=1~j=1}^{2\quad n+1}$, where $\{e_{j}\}_{j=1}^{n+1}$ is an orthonormal basis for $\mathcal{H}=\Bbb R^{n+1}$. An easy computation shows that
\begin{equation*}
S_{W}=\textnormal{diag}\bigg( \omega_{1}^{2}, \underbrace{\left( \omega_{1}^{2}+\omega_{2}^{2}\right) , \ldots ,\left( \omega_{1}^{2}+\omega_{2}^{2}\right) }_{n-1},\omega_{2}^{2}\bigg) .
\end{equation*}
Although Theorem \ref{excess new local} assures that $e(\mathcal{F}_{W})=n+1,$ we compute the excess of $\mathcal{F}_{W}$ directly by means of Proposition \ref{balan deficit}.
\begin{align*}
e(\mathcal{F}_{W})&=\sum_{i=1}^2\sum_{j=1}^{n+1}\left( 1-\omega_{i}^{2}\left\langle \pi_{W_{i}}e_{j},S_{W}^{-1}\pi_{W_{i}}e_{j} \right\rangle \right) 
\\&=2(n+1)-\bigg[ \omega_{1}^{2}\left\langle  e_{1},S_{W}^{-1}e_{1} \right\rangle  +\left( \omega_{1}^{2}+\omega_{2}^{2}\right) \left\langle  e_{2},S_{W}^{-1}e_{2} \right\rangle  +\cdots 
\\&+ \left( \omega_{1}^{2}+\omega_{2}^{2}\right) \left\langle  e_{n},S_{W}^{-1}e_{n} \right\rangle +\omega_{2}^{2}\left\langle  e_{n+1},S_{W}^{-1}e_{n+1} \right\rangle \bigg]
\\&=2(n+1)-\bigg[\omega_{1}^{2}\omega_{1}^{-2} +\left( \omega_{1}^{2}+\omega_{2}^{2}\right) \left( \omega_{1}^{2}+\omega_{2}^{2}\right) ^{-1}+\cdots 
\\&+\left( \omega_{1}^{2}+\omega_{2}^{2}\right) \left( \omega_{1}^{2}+\omega_{2}^{2}\right) ^{-1} +\omega_{2}^{2}\omega_{2}^{-2}\bigg]=2(n+1)-(n+1)=n+1.
\end{align*}
Hence, applying Theorem \ref{excess new local} we get
\begin{align*}
e(W)&=e(\mathcal{F}_{W})+\sum_{i=1}^{2}\textnormal{dim}W_{i}-2\textnormal{dim}\mathcal{H}
\\&=3n+1-2(n+1)=n-1.
\end{align*}
\subsection{Orthogonal complement fusion frames}
Through the orthogonal complement, under certain conditions, a new fusion frame is obtained from a given fusion frame; so called the orthogonal complement fusion frame \cite{existence}. Here, we discuss the excess of these fusion frames and survey the relationship between the excess of a fusion frame and its orthogonal complement fusion frame.
\begin{definition}\cite{existence}
Let $W=\{(W_{i},\omega_{i})\}_{i \in I} $ be a fusion frame for $\mathcal {H}$. If the family $W^{\perp}:=\{(W_{i}^{\perp},\omega_{i})\}_{i \in I}$, where $ W_{i}^{\perp}$ is the orthogonal complement of $W_{i}$ is also a fusion frame, then we call $W^{\perp}$ the orthogonal complement fusion frame obtained from $W$.
\end{definition}
In \cite{existence}, it has been shown that if $W=\{(W_{i},\omega_{i})\}_{i \in I} $ is a fusion frame for $\mathcal {H}$ such that $\sum_{i \in I}\omega_{i}^{2}<\infty,$ then $W^\perp$ is a fusion frame for $\mathcal {H}$ if and only if $\bigcap_{i \in I}W_{i}=\{0\}$. In what follows, without losing the generality, we assume that $\sum_{i \in I}\omega_{i}^{2}<\infty.$ Note that in the case $|I|=1$, $W^\perp=\{0\}$ is not a fusion frame. Furthermore, provided that $W$ is a fusion Riesz basis, it is easily seen that $W^\perp$ is a fusion frame. The following lemma describes the excess of the orthogonal complement fusion frame associated with a fusion Riesz basis.
\begin{lemma}\label{excess of risze}
Let $W=\{(W_{i},\omega_{i})\}_{i \in I}$ be a fusion Riesz basis for an infinite dimensional Hilbert space $\mathcal {H}$. Then the following conditions are hold.
\begin{itemize}
\item[(1)] If $|I|=2$, then $e(W^{\perp})=0$.
\item[(2)] If $|I|>2$, then $e(W^{\perp})=\infty$.
\end{itemize}
\end{lemma}

\begin{proof}
Assume that $|I|=2$. A simple computation confirms that $\bigcap_{i \in I}W_{i}^{\perp}=\{0\}$. Hence, $W$ and $W^{\perp}$ are fusion Riesz bases which assures that $e(W^{\perp})=0$. Now,
suppose that $|I|>2$ and $\{e_{i,j}\}_{j \in J_{i}}$ is an orthonormal basis of $W_{i}$ for all $i \in I$. According to Proposition \ref{mitra sh.}, for all $k \in I$, we have
\begin{equation}\label{orth property}
\begin{aligned}
W_{k}^\perp &=\overline{\textnormal{span}}_{i\neq  k} \left\lbrace S_{W}^{-1}W_{i}\right\rbrace 
\\&=\overline{\textnormal{span}}_{i\neq  k, j \in J_{i}}\left\lbrace S_{W}^{-1}e_{i,j}\right\rbrace .
\end{aligned}
\end{equation}
Note that each subspace $S_{W}^{-1}W_{i}$ occurs $|I|-1$ times in $W^{\perp}.$ Since $\left\lbrace S_{W}^{-1}e_{i,j}\right\rbrace _{j \in J_{i}}$ is a local Riesz basis for $S_{W}^{-1}W_{i}$, thus $e(W^{\perp})=\sum_{i \in I}(|I|-2)|J_{i}|$ by Proposition \ref{excess}. The assumption $\textnormal{dim} \mathcal{H}=\infty$ ensures that $|I|=\infty$ or $|J_{i}|=\infty$, for some $i \in I$. In particular, $e(W^{\perp})=\infty$.
\end{proof}
The following lemma indicates that if $W=\{(W_{i},\omega_{i})\}_{i \in I}$ is a fusion Riesz basis with $|I|\geq2,$ then having at least two subspaces of $W^\perp$ is sufficient to constitute a fusion frame.

\begin{lemma}\label{dr.dr.}
Let $W=\{(W_{i},\omega_{i})\}_{i \in I}$ be a fusion Riesz basis for $\mathcal {H}$ with $|I|\geq2.$ Then $\mathcal{W}^\perp:=\{(W_{i}^\perp,\omega_{i})\}_{i=\ell,k}$ is a fusion frame for every distinct index $\ell,k \in I.$
\end{lemma}

\begin{proof}
Assume that $W$ is a fusion Riesz basis, then $W^\perp$ is a fusion frame for $\mathcal {H}$. In order to prove that $\mathcal{W}^\perp$ is a fusion frame, it is sufficient to prove that $S_{W}^{1/2}\mathcal{W}^\perp$ is so, by \cite[Theorem 2.4]{Gavruta}. Due to the fact that $S_{W}^{1/2}\mathcal{W}^\perp$ is a fusion Bessel sequence, we just need to investigate the lower bound. To this end, we first observe that 
\begin{equation}\label{555}
S_{W}^{1/2}W_{i}^{\perp}=\overline{\textnormal{span}}_{j \in I, j\neq i} \left\lbrace S_{W}^{-1/2}W_{j}\right\rbrace , \quad (i= \ell,k)
\end{equation}
by applying \eqref{orth property}. Since $\left\lbrace \big(S_{W}^{-1/2}W_{j},1\big)\right\rbrace _{j \in I}$ is a fusion orthonormal basis \cite{frame of subspace} for $\mathcal {H}$, by employing \eqref{555} we get $\pi_{S_{W}^{1/2}W_{i}^\perp}=\sum_{j \in I, j \neq i}\pi_{S_{W}^{-1/2}W_{j}},$ for $i= \ell,k$. Thus, by taking $\gamma:=\textnormal{min}\left\lbrace \omega_{\ell}^{2},\omega_{k}^{2}\right\rbrace,$ we induce
\begin{align*}
\gamma \Vert f \Vert ^{2}&= \gamma \sum_{j \in I}\Vert \pi_{S_{W}^{-1/2}W_{j}}f\Vert ^{2}
\\&=\gamma \left(  \Vert \pi_{S_{W}^{-1/2}W_{\ell}}f\Vert ^{2}+\Vert \pi_{S_{W}^{-1/2}W_{k}}f\Vert ^{2}+\sum_{j\in I, j \neq \ell,k}\Vert \pi_{S_{W}^{-1/2}W_{j}}f\Vert ^{2}\right) 
\\&\leq \omega_{k}^{2}\Vert \pi_{S_{W}^{-1/2}W_{\ell}}f\Vert ^{2}+\omega_{\ell}^{2}\Vert \pi_{S_{W}^{-1/2}W_{k}}f\Vert ^{2}+(\omega_{k}^{2}+\omega_{\ell}^{2})\sum_{j\in I, j \neq \ell,k}\Vert \pi_{S_{W}^{-1/2}W_{j}}f\Vert ^{2}
\\&= \sum_{j\in I, j \neq k} \omega_{k}^{2}\Vert \pi_{S_{W}^{-1/2}W_{j}}f\Vert ^{2}+\sum_{j\in I, j \neq \ell} \omega_{\ell}^{2}\Vert \pi_{S_{W}^{-1/2}W_{j}}f\Vert ^{2}
\\&=\sum_{i=\ell,k}\omega_{i}^{2}\Vert \pi_{S_{W}^{1/2}W_{i}^\perp}f\Vert ^{2}.
\end{align*}
Therefore, $S_{W}^{1/2}\mathcal{W}^\perp$ is a fusion frame, as required.
\end{proof}

It should be mentioned that the converse of Lemma \ref{dr.dr.} is not generally true. For instance, one may consider the fusion frame $W=\{(W_{i},\omega_{i})\}_{i=1}^3$ of $\mathcal {H}_{3}$ where $W_{1}=\textnormal{span}\{e_{1}\},~W_{2}=\textnormal{span}\{e_{1}+e_{2}\}$ and $W_{3}=\textnormal{span}\{e_{2},e_{3}\}.$

The general connection between the excess of fusion frames $W$ and $W^\perp$ is given in the next theorem.
\begin{theorem}\label{perp th.}
Let $W=\{(W_{i},\omega_{i})\}_{i \in I}$ be a fusion frame for $\mathcal {H}$ such that $\bigcap_{i \in I}W_{i}=\{0\}$. Then
\begin{align}\label{excess perp}
e(W^{\perp})+e(W)=(|I|-2)\textnormal{dim}\mathcal {H}.
\end{align}
\end{theorem}
\begin{proof}
We split the proof into the following two cases.

$\mathbf{Case~1.}$ Assume $\textnormal{dim} \mathcal{H}=n<\infty$. Applying \eqref{finite dimensional} yields
\begin{align*}
e(W^{\perp})+e(W)&=\sum_{i\in I}\textnormal{dim}W^{\perp}_{i}+\sum_{i\in I}\textnormal{dim}W_{i}-2n
\\&=\sum_{i\in I}\big(n-\textnormal{dim}W_{i}\big)+\sum_{i\in I}\textnormal{dim}W_{i}-2n
\\&=n(|I|-2),
\end{align*}
which is the desired result.

$\mathbf{Case~2.}$ Suppose $\textnormal{dim} \mathcal{H}=\infty$. If $W$ is a fusion Riesz basis, then \eqref{excess perp} obviously holds by Lemma \ref{excess of risze}. Let $W$ be a fusion frame and $\{e_{i,j}\}_{j \in J_{i}}$ be an orthonormal basis of $W_{i}$ for each $i \in I$. In the case $e(W)=\infty$, the equation \eqref{excess perp} obviously holds. So, let $e(W)=m<\infty$ and $|I|>2$. Then, without losing the generality, we suppose that $\{e_{i_{0},j}\}_{j \in L}$ is a family of redundant vectors in $W_{i_{0}}$, where $L \subset J_{i_{0}}$ and $|L|=m$. Take
\begin{equation*}\label{vvv}
Z_{i}:=\begin{cases}
W_{i},~& i \neq i_{0}
\\
\overline{\textnormal{span}}\{e_{i_{0},j}\}_{j \in J_{i_{0}}\smallsetminus L },~& i= i_{0}.
\end{cases}
\end{equation*}
It is obviously seen that $Z=\{(Z_{i},\omega_{i})\}_{i \in I}$ is a fusion Riesz basis for $\mathcal{H}$ and $Z^{\perp}=\big\{(Z_{i}^{\perp},\omega_{i})\big\}_{i \in I}$ is a fusion frame with $e(Z^{\perp})=\infty,$ by Lemma \ref{excess of risze}. Moreover, it follows from Lemma \ref{dr.dr.} that $Z_{0}^{\perp}:=\big\{(Z_{i}^{\perp},\omega_{i})\big\}_{i \in I,i \neq i_{0}}$ is also a fusion frame for $\mathcal{H}$. According to Proposition \ref{mitra sh.}, $Z_{k}^{\perp}$ may be written as $\overline{\textnormal{span}}_{i\neq  k} \left\lbrace S_{Z}^{-1}Z_{i}\right\rbrace,~(k \in I)$. In addition,
\begin{align*}
W_{i_{0}}^{\perp}&=\left( Z_{i_{0}} + \textnormal{span}\{e_{i_{0},j}\}_{j \in L }\right) ^{\perp}
\\&=Z_{i_{0}}^{\perp}\cap \left( \textnormal{span}\{e_{i_{0},j}\}_{j \in L }\right) ^{\perp}.
\end{align*}
Hence, we get
\begin{equation*}\label{www}
W_{i}^{\perp}=\begin{cases}
\overline{\textnormal{span}}_{j\neq i} \left\lbrace S_{Z}^{-1}Z_{j}\right\rbrace ,~& i \neq i_{0}
\\
\overline{\textnormal{span}}_{j\neq i} \left\lbrace S_{Z}^{-1}Z_{j}\right\rbrace \cap \left( \textnormal{span}\{e_{i_{0},j}\}_{j \in L }\right) ^{\perp},~& i= i_{0}.
\end{cases}
\end{equation*}
Furthermore, it is not difficult to see that
\begin{equation}\label{eee}
e(Z_{0}^{\perp})\leq e(W^{\perp})\leq e(Z^{\perp}).
\end{equation}
As such, it suffices to prove that $e(Z_{0}^{\perp})=\infty$. An easy computation shows that every subspace  $S_{Z}^{-1}Z_{i}~(i \neq i_{0})$ occurs $|I|-2$ times and $S_{Z}^{-1}Z_{i_{0}}$ occurs $|I|-1$ times in $Z_{0}^{\perp}$. Applying the same argument as in Lemma \ref{excess of risze}, we have
\begin{equation*}
e(Z_{0}^{\perp})=(|I|-2)|J_{i_{0}}|+\sum_{i \in I, i \neq i_{0}}(|I|-3)|J_{i}|.
\end{equation*}
Due to $e(Z^{\perp})=\sum_{i \in I}(|I|-2)|J_{i}|=\infty$, we infer that  $|I|=\infty$ or $|J_{i}|=\infty$, for some $i \in I$. Especially, $e(Z_{0}^{\perp})=\infty$. Therefore, it follows from \eqref{eee} that $e(W^{\perp})=\infty$.
This completes the proof.
\end{proof}
The importance of Theorem \ref{perp th.} appears between the excess of $W$ and $W^\perp,$ one of which is easily computed. We now examine the validity of the above statement through two examples.
\begin{example}
\begin{itemize}
Let $\{e_{i}\}_{i=1}^n $ be an orthonormal basis for $\mathcal {H}_{n}$ and $\{\omega_{i}\}_{i \in I}$ a family of weights.
\item[(1)]
Take $W_{2i-1}=W_{2i}=\textnormal{span} \{e_{i}\},$ for $i=1,\ldots,n$. Then $W=\left\lbrace W_{2i-1},W_{2i}\right\rbrace_{i=1}^n$ is a fusion frame for $\mathcal {H}_{n}$ via $2n$ subspaces with respect to $\{\omega_{i}\}_{i \in I}$. Moreover, $W^{\perp}=\left\lbrace W_{2i-1}^{\perp},W_{2i}^{\perp}\right\rbrace $ is also a fusion frame for $\mathcal {H}_{n}$ with respect to $\{\omega_{i}\}_{i \in I}$. In view of \eqref{finite dimensional}, we get $e(W)=n$ and $e(W^{\perp})=n(2n-3)$. Hence,
\begin{align*}
e(W^{\perp})+e(W)&=n(2n-3)+n
\\&=n(2n-2).
\end{align*}
\item[(2)]
Consider $W_{i}=\textnormal{span} \left\lbrace e_{i},e_{i+1}\right\rbrace $ for each $i \in I$, where $|I|=n-1$. Then $W=\{(W_{i},\omega_{i})\}_{i \in I}$ is a fusion frame for $\mathcal {H}_{n}$ and $e(W)=n-2,$ by \eqref{finite dimensional}. In addition, $W^{\perp}=\left\lbrace (W_{i}^{\perp},\omega_{i})\right\rbrace _{i \in I}$ is a fusion frame for $\mathcal {H}_{n}$, where $W_{i}^{\perp}=\textnormal{span} \{e_{j}\}_{j\neq i,i+1}$ for each $i \in I$. So, it gives
\begin{align*}
e(W^{\perp})&=\sum_{i \in I}\textnormal{dim}W^{\perp}_{i}-\textnormal{dim}\mathcal {H}_{n}
\\&=(n-1)(n-2)-n=n^2-4n+2.
\end{align*}
Therefore,
\begin{align*}
e(W^{\perp})+e(W)&=n^2-4n+2+n-2
\\&=n\big((n-1)-2\big)=n( |I|-2) .
\end{align*}
\end{itemize}
\end{example}

\section{Excess of Dual Fusion Frames}
One of the most important properties of computing the excess in the fusion frame setting is to facilitate the construction of fusion Riesz bases, which is \hbox{provided} by the excess of their local frames. Taking this into account, our purpose is to obtain the excess of dual fusion frames. First, we recall that two fusion frames $W=\{(W_{i},\omega_{i})\}_{i \in I}$ and $Z=\{(Z_{i},\omega_{i})\}_{i \in I}$ are called  \textit{equivalent} \cite{frame of subspace} if there exists an invertible operator $U\in B(\mathcal {H})$ such that $Z_{i}=U(W_{i}),$ for all $i \in I$. It is shown that equivalent fusion frames have the same excess. Indeed, if $Z_{i}=U(W_{i})~(i \in I)$ for some invertible operator $U\in B(\mathcal {H})$, then $\{f_{i}\}_{i\in I}\in N(T_{Z}) \Leftrightarrow \{U^{-1}f_{i}\}_{i\in I}\in N(T_{W})$. Hence, there exists a bijective correspondence between $N(T_{Z})$ and $N(T_{W})$, which implies that $e(Z)=e(W)$. In particular, every fusion frame with its canonical dual have the same excess.

We now review the concept of $Q$-dual fusion frames. We denote the set of bounded left inverses of $T_{W}^{*}$ by $\mathcal {L}_{T_{W}^{*}}$. 
\begin{definition}\cite{Heineken}\label{Q}
A fusion frame $V=\{(V_{i},\upsilon_{i})\}_{i \in I}$ of $\mathcal {H}$ is called a \textit{$Q$-dual fusion frame} of $W=\{(W_{i},\omega_{i})\}_{i \in I}$ if there exists a $Q \in B\left( \bigoplus_{i \in I} W_{i},\bigoplus_{i \in I} V_{i}\right) $ such that
\begin{equation*}
T_{V}QT_{W}^{*}=I_{\mathcal {H}}.
\end{equation*}
Moreover, if the operator $M_{i}: \bigoplus_{i \in I} W_{i}\rightarrow \bigoplus_{i \in I} W_{i}$ is given by $M_{i}\{f_{j}\}_{j \in I}=\{\delta_{i,j}f_{j}\}_{j \in I}$, then 
\begin{itemize}
\item[(i)]$Q$ is called \textit{block-diagonal}, if $QM_{i}\left( \bigoplus_{i \in I} W_{i}\right) \subseteq M_{i}\left( \bigoplus_{i \in I} V_{i}\right) $ for all $i \in I$,
\item[(ii)]$Q$ is called \textit{component preserving}, if $QM_{i}\left( \bigoplus_{i \in I} W_{i}\right) =M_{i}\left( \bigoplus_{i \in I} V_{i}\right) $ for all $i \in I$.
\end{itemize}
\end{definition}
In Definition \ref{Q}, if $Q$ is block-diagonal (component preserving) we say that $V$ is \textit{block-diagonal (component preserving) dual fusion frame} of $W.$ It is easily seen that every alternate (G\~{a}vru\c{t}a) dual fusion frame $V$ of $W$ is a block-diagonal dual fusion frame with $Q=\varphi_{VW}$, which is introduced in \cite{Osgooei}. It is worthwhile to mention that, unlike ordinary frames, every fusion frame and its $Q$-dual may not necessarily have the same excess, see Examples \ref{example Q-dual} and \ref{exa last}. However, in the following proposition, we show that every fusion frame has a non canonical $Q$-dual fusion frame with the same excess. 

\begin{proposition}\label{dual Riesz non canonical}
Let $W=\{(W_{i},\omega_{i})\}_{i \in I}$ be a fusion frame of $\mathcal {H}$. Then there exists a non canonical $Q$-dual fusion frame $V=\{(V_{i},\omega_{i})\}_{i \in I}$ such that $e(V)=e(W)$.
\end{proposition}
\begin{proof}
Assume that $\mathcal{F}=\{\omega_{i}f_{i,j}\}_{i \in I, j \in J_{i}}$ is a local frame of $W,$ where $\{f_{i,j}\}_{j \in J_{i}}$ is a Riesz basis of $W_{i}$ for all $i \in I$. Consider $V_{i}=\overline{\textnormal{span}}\left\lbrace \omega_{i}S_{\mathcal{F}}^{-1}f_{i,j}\right\rbrace _{j \in J_{i}},$ for all $i \in I$ and take
\begin{equation*}
Q: \bigoplus_{i \in I} W_{i}\rightarrow \bigoplus_{i \in I} V_{i},~~Q\{h_{i}\}_{i \in I}=\left\lbrace \sum_{j \in J_{i}}\langle h_{i}, f_{i,j}\rangle S_{\mathcal{F}}^{-1}f_{i,j}\right\rbrace _{i \in I}.
\end{equation*}
Since $\left\lbrace \omega_{i}S_{\mathcal{F}}^{-1}f_{i,j}\right\rbrace _{i \in I, j \in J_{i}}$ is the canonical dual of $\mathcal{F},$ we conclude that $V=\{(V_{i},\omega_{i})\}_{i \in I}$ is a $Q$-dual fusion frame of $W,$ by \cite[Theorem 3.12]{Heineken}. Therefore, the result is obtained from Proposition \ref{excess} and noting the fact that dual frames have the same excess.
\end{proof}

Suppoes that $W=\{(W_{i},\omega_{i})\}_{i \in I}$ is a fusion Riesz basis of $\mathcal {H}$. Although  \cite[Theorem 2.9]{Arabyani dual} proved that $\widetilde{W}$ is the only G\~{a}vru\c{t}a dual fusion Riesz basis of $W,$ Proposition 4.2 presents a non canonical $Q$-dual fusion Riesz basis of $W$.
\begin{corollary}
Every fusion Riesz basis has a non canonical $Q$-dual fusion Riesz basis.
\end{corollary}

By Lemmata 3.4 and 3.5 in \cite{Heineken}, it is proved that if $V=\{(V_{i},\upsilon_{i})\}_{i \in I}$ is a block-diagonal dual fusion frame of $W,$ then $\mathcal{A}M_{i}\left( \bigoplus_{i \in I} W_{i}\right) \subseteq V_{i}$, for each  $i\in I$, where $\mathcal{A} \in \mathcal {L}_{T_{W}^{*}}$. Moreover, the converse implication holds under some conditions. More precisely, if $\mathcal{A}M_{i}\left( \bigoplus_{i \in I} W_{i}\right)\subseteq V_{i}$ such that $V$ is a fusion Bessel sequence and 
\begin{equation*}
Q_{\mathcal{A}}: \bigoplus_{i \in I} W_{i}\rightarrow \bigoplus_{i \in I} V_{i},~~Q_{\mathcal{A}}\{f_{j}\}_{j \in I}=\left\lbrace \upsilon_{i} ^{-1}\mathcal{A}M_{i}\{f_{j}\}_{j \in I}\right\rbrace _{i \in I},
\end{equation*}
is a well-defined bounded operator, then $V$ is a $Q_{\mathcal{A}}$-block-diagonal dual fusion frame of $W.$ Hence, block-diagonal dual fusion frames are linked to the left inverses of the analysis operator $ T_{W}^{*} $. In what follows, by the $Q$-dual fusion frame we mean the $Q$-block-diagonal dual fusion frame. Furthermore, the linear transformation $p_{i}:\bigoplus_{i \in I} W_{i}\rightarrow W_{i}$ is defined by $p_{i}(\{f_{k}\}_{k \in I})=f_{i}$ for each $i \in I$. Given a fusion frame, we first derive a bound for the excess of its $Q$-dual fusion frame.
\begin{theorem}\label{bounds}
Let $W=\{(W_{i},\omega_{i})\}_{i \in I}$ be a fusion frame for $\mathcal {H}_{n}$ with a $Q_{\mathcal{A}}$-dual fusion frame $V=\{(V_{i},\upsilon_{i})\}_{i \in I},$ where $\mathcal{A} \in \mathcal {L}_{T_{W}^{*}}$. Then
\begin{equation*}
\big |e(V)-e(W)\big |=\bigg |\sum_{i \in I} \textnormal{dim}\big( V_{i}\cap N(p_{i}\mathcal{A}^{*})\big)-\sum_{i \in I} \textnormal{dim}N(\mathcal{A}p_{i}^{*})\bigg |.
\end{equation*}
Moreover, if $V$ is a component preserving dual fusion frame, then the right-hand side of the above identity can be summarized to
\begin{equation*}
\big |e(V)-e(W)\big |=\big |\textnormal{dim}N(Q_{\mathcal{A}}^{*})-\textnormal{dim}N(Q_{\mathcal{A}})\big |.
\end{equation*}
\end{theorem}
\begin{proof}
Since $V$ is a $Q_{\mathcal{A}}$-dual fusion frame of $W,$ we observe that
\begin{equation*}
V_{i}^{\perp}\cap R(\mathcal{A}p_{i}^{*})\subseteq V_{i}^{\perp}\cap V_{i}=\{0\},~(i \in I).
\end{equation*}
Therefore, we get
\begin{align*}
\bigg |\sum_{i \in I} \textnormal{dim}\big(V_{i}&\cap N(p_{i}\mathcal{A}^{*})\big)-\sum_{i \in I} \textnormal{dim}N(\mathcal{A}p_{i}^{*})\bigg |
\\&=\bigg |\sum_{i \in I} \textnormal{dim}\left(V_{i}\cap R(\mathcal{A}p_{i}^{*})^\perp\right)-\sum_{i \in I} \textnormal{dim}N(\mathcal{A}p_{i}^{*})\bigg |
\\&=\bigg |\sum_{i \in I} \textnormal{dim}\left(V_{i}^\perp + R(\mathcal{A}p_{i}^{*})\right)^\perp-\sum_{i \in I} \textnormal{dim}N(\mathcal{A}p_{i}^{*})\bigg |
\\&=\bigg |\sum_{i \in I} \textnormal{dim}\mathcal {H}_{n}-\sum_{i \in I} \textnormal{dim}\left(V_{i}^\perp + R(\mathcal{A}p_{i}^{*})\right)-\sum_{i \in I} \textnormal{dim}N(\mathcal{A}p_{i}^{*})\bigg |
\\&=\bigg |\sum_{i \in I}\textnormal{dim}V_{i} - \sum_{i \in I} \textnormal{dim}R(\mathcal{A}p_{i}^{*}) -\sum_{i \in I} \textnormal{dim}N(\mathcal{A}p_{i}^{*})\bigg |
\\&=\bigg |\sum_{i \in I}\textnormal{dim}V_{i}-\sum_{i \in I}\textnormal{dim}W_{i}\bigg |=\big |e(V)-e(W)\big |.
\end{align*}
For the moreover part, if $V$ is a component preserving dual fusion frame, then the desired result is obtained from the fact that
\begin{align*}
N(Q_{\mathcal{A}}^{*})&=\left\lbrace \{g_{i}\}_{i \in I}\in \bigoplus_{i \in I}V_{i}:~ \sum_{i\in I}\upsilon _{i}^{-1}M_{i}\mathcal{A}^{*}g_{i}=0\right\rbrace 
\\&=\left\lbrace \{g_{i}\}_{i \in I}\in \bigoplus_{i \in I}V_{i}:~ p_{i}\mathcal{A}^{*}g_{i}=0,~ \forall ~ i \in I \right\rbrace 
\\&=\bigoplus_{i \in I}\big( V_{i}\cap N(p_{i}\mathcal{A}^{*})\big), 
\end{align*}
and similarly $N(Q_{\mathcal{A}})=\bigoplus_{i \in I}N(\mathcal{A}p_{i}^{*})$.
\end{proof}\color{black}
As a special case of Theorem \ref{bounds}, we get the following result.
\begin{corollary}\label{Gavruta bound}
Assume that $V,$ considered in Theorem \ref{bounds}, is a G\~{a}vru\c{t}a dual of $W,$ then
\begin{equation}\label{ssss}
\begin{aligned}
\big |e(V)-e(W)\big |&=\bigg |\sum_{i \in I} \textnormal{dim}\big(V_{i}\cap S_{W}W_{i}^{\perp}\big) -\sum_{i \in I} \textnormal{dim}\big(V_{i}+S_{W}W_{i}^{\perp}\big)^{\perp}\bigg |
\\&= \big |\textnormal{dim}N(\varphi_{VW}^{*})-\textnormal{dim}N(\varphi_{VW})\big |.
\end{aligned}
\end{equation}
Moreover, if $V_{i}\supseteq \widetilde{W}_{i}$ for all $i\in I$, then 
\begin{align*}
\big |e(V)-e(W)\big |=\sum_{i \in I} \textnormal{dim}\left( V_{i}\cap S_{W}W_{i}^{\perp}\right) .
\end{align*}
\end{corollary}
\begin{proof}
Due to $V$ is a G\~{a}vru\c{t}a dual of $W,$ we have $\mathcal{A}:=T_{V}\varphi_{VW} \in \mathcal {L}_{T_{W}^{*}}$. Hence,
\begin{align*}
\mathcal{A}p_{i}^{*}\pi_{W_{i}}f&=T_{V}\varphi_{VW}p_{i}^{*}\pi_{W_{i}}f
\\&=T_{V}\varphi_{VW}\left\lbrace \delta_{k,i}\pi_{W_{k}}f\right\rbrace _{k \in I}
\\&=T_{V}\left\lbrace \delta_{k,i}\pi_{V_{k}}S_{W}^{-1}\pi_{W_{k}}f\right\rbrace _{k \in I}
\\&=\upsilon_{i} \pi_{V_{i}}S_{W}^{-1}\pi_{W_{i}}f,
\end{align*}
and similarly $p_{i}\mathcal{A}^{*}f=\upsilon_{i} \pi_{W_{i}}S_{W}^{-1}\pi_{V_{i}}f,$ for all $f \in \mathcal {H}$. Therefore, \eqref{ssss} follows from Theorem \ref{bounds} and noting the fact that
\begin{align*}
\sum_{i \in I} \textnormal{dim}\big(W_{i}\cap S_{W}V_{i}^{\perp}\big)&=\sum_{i \in I} \textnormal{dim}S_{W}\left( V_{i}^{\perp}\cap S_{W}^{-1}W_{i}\right) 
\\&=\sum_{i \in I} \textnormal{dim}\left(V_{i}^{\perp}\cap \big(S_{W}W_{i}^{\perp}\big)^{\perp}\right) 
\\&=\sum_{i \in I} \textnormal{dim}\left(V_{i}+S_{W}W_{i}^{\perp}\right)^{\perp}.
\end{align*}
For the moreover part, if $ V_{i}\supseteq \widetilde{W}_{i}$ for all $i\in I$, then it follows
\begin{equation*}
\sum_{i \in I} \textnormal{dim}\left( V_{i}+S_{W}W_{i}^{\perp}\right)^{\perp}=\sum_{i \in I} \textnormal{dim}\left( V_{i}^{\perp}\cap \widetilde{W}_{i}\right)=0,
\end{equation*}
which completes the proof. 
\end{proof}
\begin{Remark}\label{Riesz dual}
Suppose that $W$ is a fusion Riesz basis for $\mathcal {H}_{n}$. In \cite{Arabyani dual}, it has been proved that a Bessel sequence $V=\{(V_{i},\omega_{i})\}_{i \in I}$ is a G\~{a}vru\c{t}a dual fusion frame of $W$ if and only if $ V_{i}\supseteq \widetilde{W}_{i},$ for all $i\in I$. Hence, Corollary \ref{Gavruta bound} ensures that $e(V)=\sum_{i \in I} \textnormal{dim}\big(V_{i}\cap S_{W}W_{i}^{\perp}\big)$ for every G\~{a}vru\c{t}a dual $V$ of $W$. Overall, it can be concluded that
\begin{equation*} 
0 \leq e(V) \leq \sum_{i \in I} \textnormal{dim}W_{i}^\perp .
\end{equation*}
In particular, in the case $|I|=2$, the excess of every G\~{a}vru\c{t}a dual is $\leq \textnormal{dim}\mathcal {H}_{n},$ by Lemma \ref{excess of risze}.
\end{Remark}

\begin{example}\label{poi}
\begin{itemize}
\item[(1)] Take $W_{1}=\textnormal{span} \{(1,1,0)\}$ and $W_{2}=\{0\}\times \Bbb R^{2}.$ Then $W=\{(W_{i},\omega_{i})\}_{i=1}^2$ is a fusion Riesz basis for $\mathcal {H}=\Bbb R^{3} $. An easy computation shows that
\begin{equation*}
S_{W}=\frac{1}{2}
\begin{pmatrix}
   \omega_{1}^{2}  &  \omega_{1}^{2} & 0\\
    \omega_{1}^{2}  &  \omega_{1}^{2}+ 2\omega_{2}^{2} & 0 \\
   0  & 0  & 2\omega_{2}^{2}
\end{pmatrix}.
\end{equation*}
Moreover, $S_{W}W_{1}^{\perp}=W_{2}$ and $S_{W}W_{2}^{\perp}=W_{1}.$ Now consider the G\~{a}vru\c{t}a dual fusion frame $V=\{(V_{i},\omega_{i})\}_{i=1}^2$ of $W,$ where $V_{1}=\Bbb R^{2}\times\{0\}$ and $V_{2}=\mathcal {H}$. By Remark \ref{Riesz dual}, it implies that $e(V)=\sum_{i=1}^{2} \textnormal{dim}\big(V_{i}\cap S_{W}W_{i}^{\perp}\big)=2.$

\item[(2)] Consider the fusion frame $W=\{(W_{i},\omega)\}_{i=1}^2$ introduced in Example \ref{excess of fusion frames}(4) with $e(W)=1$. Take $V_{1}=\overline{\textnormal{span}} \{e_{i}\}_{i\geq -n}$ and $V_{2}=\overline{\textnormal{span}} \{e_{i}\}_{i\leq m}$, for every $n,m \in \Bbb{N}$. Then $V=\{(V_{i},\omega)\}_{i=1}^2$ is a G\~{a}vru\c{t}a dual fusion frame of $W$ such that
$S_{V}^{-1}f=\omega^{-2}\left( f-\sum_{i=-n}^{m}\langle f,e_{i} \rangle \frac{e_{i}}{2}\right) $, for all $f \in \mathcal {H}$. It follows from Theorem \ref{th} that
\begin{align*}
e(V)&=\sum_{i\in \Bbb Z}\left( 1-\omega^{2}\left\langle  e_{i},S_{V}^{-1}e_{i} \right\rangle \right) +\sum_{i=-n}^{m}\left( 1-\omega^{2}\left\langle  e_{i},S_{V}^{-1}e_{i}\right\rangle \right) 
\\&=\sum_{i\in \Bbb Z}\sum_{j=-n}^{m}\left\langle  e_{i},\frac{e_{j}}{2}\right\rangle +\sum_{i=-n}^{m}\sum_{j=-n}^{m}\left\langle  e_{i},\frac{e_{j}}{2}\right\rangle =n+m+1.
\end{align*}
Furthermore, it is not difficult to see that $V_{i}\supseteq \widetilde{W}_{i},$ for all $i\in I$ and 
\begin{equation*}
\sum_{i=1}^2 \textnormal{dim}\big(V_{i}\cap S_{W}W_{i}^{\perp}\big)=n+m.
\end{equation*} 
Thus, we derive
\begin{align*}
e(V)-e(W)&=(n+m+1)-1 
\\&=\sum_{i=1}^2 \textnormal{dim}\big(V_{i}\cap S_{W}W_{i}^{\perp}\big).
\end{align*}

\item[(3)] Consider the fusion frame $W=\{(W_{i},\omega)\}_{i=1}^\infty$ introduced in Example \ref{excess of fusion frames}(3) with the fusion frame operator
\begin{equation*}
S_{W}=\textnormal{diag}\bigg( \omega ^{2},\underbrace{2\omega ^{2}, 2\omega ^{2},\ldots, 2\omega ^{2}}_{n-1}, \omega ^{2}, \omega ^{2}, \ldots \bigg) .
\end{equation*} 
Take $m\in \Bbb{N}$ and
\begin{equation*}
V_{i}=\begin{cases}
\textnormal{span}\{e_{i},\ldots ,e_{i+m}\},~& 1 \leq i \leq n,
\\
\textnormal{span}\{e_{i+1}\},~& i > n.
\end{cases}
\end{equation*}
It is easily seen that $V=\{(V_{i},\omega)\}_{i=1}^\infty$ is a G\~{a}vru\c{t}a dual fusion frame of $W$ such that $V_{i}\supseteq \widetilde{W}_{i},$ for all $i\in I$. Thus, it follows from Corollary \ref{Gavruta bound} that
\begin{align*}
e(V)&=\sum_{i \in I} \textnormal{dim}\big(V_{i}\cap S_{W}W_{i}^{\perp}\big)+e(W)
\\&=n(m-1)+(n-1)=nm-1.
\end{align*}
\end{itemize}
\end{example}
\begin{example}
Assume that $\{e_{i}\}_{i=1}^n$ is an orthonormal basis for $\mathcal {H}_{n}$, $n\geq 2$ and $\{\omega_{i}\}_{i \in I}, \{\upsilon_{i}\}_{i \in I}$ are family of weights. Consider $W_{2i-1}=W_{2i}=\textnormal{span}\{e_{i},e_{i+1}\}$ for $i=1,\ldots,n-1$. Then $W=\left\lbrace W_{2i-1},W_{2i}\right\rbrace_{i=1}^{n-1}$ is a fusion frame for $\mathcal {H}_{n}$ with respect to $\{\omega_{i}\}_{i \in I}$ such that $e(W)=3n-4,$ by \eqref{finite dimensional}. Moreover, $V=\left\lbrace W_{2i-1},W_{2i}^{\perp}\right\rbrace_{i=1}^{n-1}$ is a G\~{a}vru\c{t}a dual fusion frame of $W$ with respect to $\{\upsilon_{i}\}_{i \in I}$. Straightforward calculations show that
\begin{align*}
\sum_{i=1}^{n-1}\sum_{j=0}^1 \textnormal{dim}\left(V_{2i-j}\cap S_{W}W_{2i-j}^{\perp}\right) &=\sum_{i=1}^{n-1} \textnormal{dim}\left( W_{2i}^{\perp}\cap S_{W}W_{2i}^{\perp}\right) 
\\&=(n-1)(n-2),
\end{align*}
and similarly $\sum_{i=1}^{n-1}\sum_{j=0}^1 \textnormal{dim}\big(V_{2i-j}+S_{W}W_{2i-j}^{\perp}\big)^{\perp}=2(n-1)$. Therefore, it follows from Corollary \ref{Gavruta bound} that
\begin{align*}
e(V)&=(n-1)(n-2)-2(n-1)+e(W)
\\&=(n-1)(n-4)+3n-4=n(n-2).
\end{align*}
\end{example}

\begin{theorem}
Let $W=\{(W_{i},\omega_{i})\}_{i \in I}$ be a fusion frame for $\mathcal {H}_{n}$ with a $Q_{\mathcal{A}}$-dual fusion frame $V=\{(V_{i},\upsilon_{i})\}_{i \in I},$ where $\mathcal{A} \in \mathcal {L}_{T_{W}^{*}}$. Then
\begin{align*}
e(V)+\textnormal{dim} R(Q_{\mathcal{A}}T_{W}^{*}T_{V})&=\textnormal{dim}\big( N(Q_{\mathcal{A}}^{*})\cap N(T_{W}Q_{\mathcal{A}}^{*})\big)  
\\&+\textnormal{dim}\big( R(Q_{\mathcal{A}}^{*})\cap N(T_{W})\big)  .
\end{align*}
\end{theorem}
\begin{proof}
Since $V$ is a $Q_{\mathcal{A}}$-dual fusion frame of $W,$ thus $T_{V}Q_{\mathcal{A}}T_{W}^{*}=I_{\mathcal {H}}.$ Applying Lemma 2.1 in \cite{Bakic} implies that 
\begin{equation*}
N(T_{V})=(I_{\oplus V_{i}}-Q_{\mathcal{A}}T_{W}^{*}T_{V})N(T_{W}Q_{\mathcal{A}}^{*}).
\end{equation*}
Moreover, using the rank-nullity theorem, we have
\begin{align*}
\textnormal{dim}N(T_{W}Q_{\mathcal{A}}^{*})&=\textnormal{dim}(I_{\oplus V_{i}}-Q_{\mathcal{A}}T_{W}^{*}T_{V})N(T_{W}Q_{\mathcal{A}}^{*}) 
\\&+\textnormal{dim}N(I_{\oplus V_{i}}-Q_{\mathcal{A}}T_{W}^{*}T_{V}).
\end{align*}
Therefore, $e(V)+\textnormal{dim}R(Q_{\mathcal{A}}T_{W}^{*}T_{V})=\textnormal{dim}N(T_{W}Q_{\mathcal{A}}^{*})$. On the other hand,
\begin{align*}
\textnormal{dim}N(T_{W}Q_{\mathcal{A}}^{*})&=\textnormal{dim}N\left( Q_{\mathcal{A}}^{*}\restriction_{N(T_{W}Q_{\mathcal{A}}^{*})}\right) + \textnormal{dim}R\left( Q_{\mathcal{A}}^{*}\restriction_{N(T_{W}Q_{\mathcal{A}}^{*})}\right) 
\\&= \textnormal{dim}\big(N(Q_{\mathcal{A}}^{*})\cap N(T_{W}Q_{\mathcal{A}}^{*})\big)+\textnormal{dim}\big(R(Q_{\mathcal{A}}^{*})\cap N(T_{W})\big).
\end{align*}
And consequently we get
\begin{align*}
e(V)+\textnormal{dim}R(Q_{\mathcal{A}}T_{W}^{*}T_{V})&=\textnormal{dim}N(T_{W}Q_{\mathcal{A}}^{*})
\\&= \textnormal{dim}\big(N(Q_{\mathcal{A}}^{*})\cap N(T_{W}Q_{\mathcal{A}}^{*})\big)
\\&+\textnormal{dim}\big(R(Q_{\mathcal{A}}^{*})\cap N(T_{W})\big).
\end{align*}
\end{proof}

Now, let us restrict our attention to the G\~{a}vru\c{t}a dual fusion frames. We obtain the next result, which is proved in a similar way.
\begin{corollary}
Let $W=\{(W_{i},\omega_{i})\}_{i \in I}$ be a fusion frame for $\mathcal {H}_n$ with a G\~{a}vru\c{t}a dual $V=\{(V_{i},\upsilon_{i})\}_{i \in I}$. Then
\begin{align*}
e(V)+\textnormal{dim}R(\varphi_{VW}T_{W}^{*}T_{V})&= \textnormal{dim}\big(N(\varphi_{VW}^{*})\cap N(T_{W}\varphi_{VW}^{*})\big)
\\&+\textnormal{dim}\big(R(\varphi_{VW}^{*})\cap N(T_{W})\big).
\end{align*}
\end{corollary}
Motivated by Theorem \ref{th}, we obtain the following explicit formula for computing the excess of component preserving dual fusion frames.
\begin{theorem}\label{excess Q dual}
Let $W=\{(W_{i},\omega_{i})\}_{i \in I}$ be a fusion frame for $\mathcal {H}$ with a $Q_{\mathcal{A}}$-component preserving dual fusion frame $V=\{(V_{i},\upsilon_{i})\}_{i \in I},$ where $\mathcal{A} \in \mathcal {L}_{T_{W}^{*}}$ and $\textnormal{dim} V_{i}<\infty$ for all $i \in I$. Then
\begin{equation*}
e(V)=\sum_{i\in I}\big(\textnormal{dim} V_{i}-\omega_{i}\textnormal{trace}\left( \mathcal{A}p_{i}^{*}\pi_{W_{i}}\right)\big) .
\end{equation*}
\end{theorem}
\begin{proof}
Let $\{u_{i,j}\}_{j \in J_{i}}$ be an orthonormal basis for $V_{i}~(i\in I)$ and $\{U_{i,j}\}_{i\in I, j \in J_{i}}$ be the orthonormal basis defined as in \eqref{oonnbb} for $\bigoplus_{i \in I} V_{i}$. The operator $P$ of $\bigoplus_{i \in I} V_{i}$ onto $N(T_{V})$ defined by $P=I_{\bigoplus V_{i}}-Q_{\mathcal{A}}T_{W}^{*}T_{V}$ is idempotent. Therefore, we obtain
\begin{align*}
e(V)=\textnormal{dim}N(T_{V})&=\textnormal{trace}(P)
\\&=\sum_{i\in I,j\in J_{i}}\big\langle U_{i,j},PU_{i,j} \big\rangle
\\&=\sum_{i\in I,j\in J_{i}}\left( 1-\big\langle U_{i,j},Q_{\mathcal{A}}T_{W}^{*}T_{V}U_{i,j} \big\rangle \right) 
\\&=\sum_{i\in I,j\in J_{i}}\left( 1-\big\langle U_{i,j},Q_{\mathcal{A}}\{\omega_{k}\pi_{W_{k}}\upsilon_{i}u_{i,j}\}_{k\in I} \big\rangle \right) 
\\&=\sum_{i\in I,j\in J_{i}}\left( 1-\left\langle  U_{i,j},\big\{\upsilon_{i}^{-1}\mathcal{A}M_{i}\{\omega_{k}\pi_{W_{k}}\upsilon_{i}u_{i,j}\}_{k\in I}\big\}_{i \in I} \right\rangle  \right) 
\\&=\sum_{i\in I,j\in J_{i}}\left( 1-\big\langle u_{i,j},\mathcal{A}\{\delta_{k,i}\omega_{k}\pi_{W_{k}}u_{i,j}\}_{k\in I} \big\rangle \right) 
\\&=\sum_{i\in I,j\in J_{i}}\left( 1-\big\langle u_{i,j},\mathcal{A}p_{i}^{*}\omega_{i}\pi_{W_{i}}u_{i,j} \big\rangle \right) 
\\&=\sum_{i\in I}\bigg( \textnormal{dim} V_{i}-\omega_{i}\sum_{j\in J_{i}}\big\langle u_{i,j},\mathcal{A}p_{i}^{*}\pi_{W_{i}}u_{i,j} \big\rangle \bigg) 
\\&=\sum_{i\in I}\big( \textnormal{dim} V_{i}-\omega_{i}\textnormal{trace}(\mathcal{A}p_{i}^{*}\pi_{W_{i}})\big) .
\end{align*}
\end{proof}
In the sequel, we present a component preserving dual fusion frame of the fusion frame introduced in Example \ref{555}(2) and subsequently compute its excess by employing the previous theorem.
\begin{example}\label{example Q-dual}
Consider the 2-equi-dimensional fusion frame $W=\{(W_{i},\omega)\}_{i \in I}$ introduced in Example \ref{555}(2) with the fusion frame operator
\begin{equation*}
S_{W}=\textnormal{diag}\big(\omega ^{2}, 2\omega ^{2}, 2\omega ^{2},\ldots \big).
\end{equation*}
Notice that
\begin{equation*}
\bigoplus_{i \in I}W_{i}=\left\lbrace \bigg\{ \sum_{k=1}^{2}c_{j,k}e_{j+k-1}\bigg \} _{j \in I};~ \sum_{j \in I}\sum_{k=1}^2\vert c_{j,k}\vert ^2<\infty \right\rbrace .
\end{equation*}
Every $\mathcal{A} \in \mathcal {L}_{T_{W}^{*}}$ is given by $S_{W}^{-1}T_{W}+R,$ where $R \in B\left( \bigoplus_{i \in I} W_{i},\mathcal {H}\right) $ and $RT_{W}^{*}=0$. Hence, the matrix representation of $\mathcal{A}$ is of the form
\begin{equation*}
\mathcal{A}=
\begin{pmatrix}
\omega ^{-1}  &  r_{1,2}  &  r_{1,3}  &  r_{1,4}  &  r_{1,5} &  \ldots\\
 0   &  \frac{\omega ^{-1}}{2}+r_{2,2} & \frac{\omega ^{-1}}{2}+r_{2,3} &  r_{2,4}  &  r_{2,5}  &  \ldots\\
 0 & r_{3,2}  &  r_{3,3}  &  \frac{\omega ^{-1}}{2}+r_{3,4} & \frac{\omega ^{-1}}{2}+r_{3,5} &  \ldots\\
 0 & r_{4,2}  &  r_{4,3}  &  r_{4,4}  &  r_{4,5} &  \ddots \\
 \vdots & \vdots  &  \vdots  &  \vdots  &  \vdots &  \vdots \\
\end{pmatrix},
\end{equation*}
where $\omega r_{i,j}+\omega r_{i,j+1}=0$ for each $i \in \Bbb{N}$ and even number $j\in \Bbb{N}$. Consider $\mathcal{B}$ as follows.
\begin{equation*}
\mathcal{B}=
{\left(\begin{smallmatrix}
\omega ^{-1}  &  0  &  0  &  0  &  0 &  0  &  0 & \ldots\\
 0   &  \frac{\omega ^{-1}}{2}+r_{2,2} & \frac{\omega ^{-1}}{2}+r_{2,3} &  0  & 0  & 0  &  0 & \ldots\\
 0 & 0  &  0  &  \ddots & 0 &  0  &  0 & \ldots\\
 0 & 0  &  0  &  \frac{\omega ^{-1}}{2}+r_{n,2n-2}  &  \frac{\omega ^{-1}}{2}+r_{n,2n-1} & 0 &  0 &  \ldots \\
 0 & 0  &  0  &  0  &   0  &  \omega ^{-1} & 0 & \ldots \\
  0 & 0  &  0  &  0  &   0  &  0 & 0 & \scriptsize\ddots \\
 \scriptsize\vdots & \scriptsize\vdots  &  \scriptsize\vdots  &  \scriptsize\vdots  &  \scriptsize\vdots &  \scriptsize\vdots &  \scriptsize\vdots  \\
\end{smallmatrix}\right)},
\end{equation*}
where $\omega r_{i,2i-2}+\omega r_{i,2i-1}=0$ for $i=2, \ldots ,n$ and $r_{i,2i-2}=-r_{i,2i-1}=\frac{\omega^{-1}}{2}$ for all $i>n$. Then $\mathcal{B} \in \mathcal {L}_{T_{W}^{*}}$ and 
\begin{equation*}
V_{i}:=\mathcal{B}M_{i}\bigg( \bigoplus_{i \in I}W_{i}\bigg) =\begin{cases}
\textnormal{span}\{e_{i},e_{i+1}\},~& 1 \leq i \leq n,
\\
\textnormal{span}\{e_{i+1}\},~& i > n.
\end{cases}
\end{equation*}
So $V=\{(V_{i},\upsilon_{i})\}_{i \in I}$ constitutes a fusion frame for $\mathcal {H}$ and \hbox{$Q_{\mathcal{B}}: \bigoplus_{i \in I} W_{i}\rightarrow \bigoplus_{i \in I} V_{i},$}
\begin{equation*}
Q_{\mathcal{B}}\left\lbrace \sum_{k=1}^{2}c_{j,k}e_{j+k-1}\right\rbrace _{j \in I}=\left\lbrace \bigg\{\sum_{j=1}^{2}\upsilon_{i}^{-1}c_{i,j}e_{i+j-1}\bigg \}_{i=1}^n,\bigg \{\upsilon_{i} ^{-1}c_{i+1,i+1}e_{i+1}\bigg \}_{i=n+1}^{\infty}\right\rbrace ,
\end{equation*}
is a well-defined bounded operator. Hence, $V$ is a $Q_{\mathcal{B}}$-component preserving dual fusion frame of $W,$ by \cite[Lemma 3.5]{Heineken}. Now, we intend to calculate the excess of the fusion frame $V.$ Note that
\begin{equation*}
\sum_{i=n+1}^{\infty}\big(\textnormal{dim} V_{i}-\omega\textnormal{trace}(\mathcal{B}p_{i}^{*}\pi_{W_{i}})\big)=\sum_{i=n+1}^{\infty}\left( 1-\omega \omega ^{-1}\right) =0.
\end{equation*}
Therefore, by applying Theorem \ref{excess Q dual}, we derive
\begin{align*}
e(V)&=\sum_{i=1}^{n}\big(\textnormal{dim} V_{i}-\omega\textnormal{trace}(\mathcal{B}p_{i}^{*}\pi_{W_{i}})\big)
\\&=\left[ 2-\omega \left( \omega ^{-1}+\frac{\omega ^{-1}}{2}+r_{2,2}\right) \right] +\left[ 2-\omega \left( \omega ^{-1}+\frac{\omega ^{-1}}{2}+r_{n,2n-1}\right) \right] 
\\&+\sum_{i=2}^{n-1}\left[ 2-\omega\left( \omega ^{-1}+r_{i,2i-1}+r_{i+1,2i}\right) \right]
\\&=2n-\left[ (n+1)+\sum_{i=2}^{n}(\omega r_{i,2i-2}+\omega r_{i,2i-1})\right] =2n-(n+1)=n-1.
\end{align*}
\end{example}

The next corollary is an immediate result of Theorem \ref{excess Q dual}.
\begin{corollary}\label{gavruta excess dual}
Let $W=\{(W_{i},\omega_{i})\}_{i \in I}$ be a fusion frame for $\mathcal {H}$ with a G\~{a}vru\c{t}a dual fusion frame $V=\{(V_{i},\upsilon_{i})\}_{i \in I}$ such that $\textnormal{dim} V_{i}<\infty,$ for all $i \in I$. Then
\begin{equation*}
e(V)=\sum_{i\in I}\left( \textnormal{dim} V_{i}-\omega_{i}\upsilon_{i}\textnormal{trace}\left( \pi_{V_{i}}S_{W}^{-1}\pi_{W_{i}}\right) \right) .
\end{equation*}
\end{corollary}
Suppose that $W=\{(W_{i},\omega_{i})\}_{i \in I}$ is a fusion frame for $\mathcal {H}$. As already mentioned, every bounded left inverse of $T_{W}^{*}$ is the operator $\mathcal{A}$ of the form $\mathcal{A}=S_{W}^{-1}T_{W}+RP,$ where $R \in B\left( \bigoplus_{i \in I} W_{i},\mathcal {H}\right) $ and $P=I_{\bigoplus W_{i}}-T_{W}^{*}S_{W}^{-1}T_{W}$ is an orthogonal projection from $\bigoplus_{i \in I} W_{i}$ onto $N(T_{W})$. Hence, for every $Q_{\mathcal{A}}$-dual fusion frame $V=\{(V_{i},\omega_{i})\}_{i \in I}$ of $W$ such that $\textnormal{dim} V_{i}<\infty$ for all $i \in I$, it follows
\begin{align*}
e(V)&=\sum_{i\in I}\big(\textnormal{dim} V_{i}-\omega_{i}\textnormal{trace}(\mathcal{A}p_{i}^{*}\pi_{W_{i}})\big)
\\ &=\sum_{i\in I}\left( \textnormal{dim} V_{i}-\omega_{i}\textnormal{trace}\left( S_{W}^{-1}T_{W}p_{i}^{*}\pi_{W_{i}}+RPp_{i}^{*}\pi_{W_{i}}\right) \right) 
\\ &=\sum_{i\in I}\left( \textnormal{dim} V_{i}-\omega_{i}^{2}\textnormal{trace}\left( \pi_{W_{i}}S_{W}^{-1}\pi_{W_{i}}\right) -\omega_{i}\textnormal{trace}\left( RPp_{i}^{*}\pi_{W_{i}}\right) \right) 
\\ &=\sum_{i\in I}\big(\textnormal{dim} V_{i}-\textnormal{dim} W_{i}\big)+e(W)-\sum_{i\in I}\omega_{i}\textnormal{trace}\left( RPp_{i}^{*}\pi_{W_{i}}\right) .
\end{align*}

Here, we compute the excess of a G\~{a}vru\c{t}a dual fusion frame of a given fusion frame.
\begin{example}\label{exa last}
Consider the fusion frame $W=\{(W_{i},\omega)\}_{i=1}^\infty$ and its G\~{a}vru\c{t}a dual $V=\{(V_{i},\omega)\}_{i=1}^\infty$ introduced in Example \ref{poi}(3). As we have already seen $e(V)=nm-1.$ Now, we are going to calculate the excess of $V$ by Corollary \ref{gavruta excess dual}. To this end, note that
\begin{equation*}
\sum_{i=n+1}^{\infty}\left( \textnormal{dim} V_{i}-\omega ^{2}\textnormal{trace}\left( \pi_{V_{i}}S_{W}^{-1}\pi_{W_{i}}\right) \right) =\sum_{i=n+1}^{\infty}\big(1-\omega ^{2}\omega ^{-2}\big)=0.
\end{equation*}
Hence, we get
\begin{align*}
e(V)&=\sum_{i=1}^\infty\left( \textnormal{dim} V_{i}-\omega ^{2}\textnormal{trace}\left( \pi_{V_{i}}S_{W}^{-1}\pi_{W_{i}}\right) \right) 
\\&=\sum_{i=1}^{n}\left( \textnormal{dim} V_{i}-\omega ^{2}\textnormal{trace}\left( \pi_{V_{i}}S_{W}^{-1}\pi_{W_{i}}\right) \right) 
\\&=2\left[ (m+1)-\omega ^{2}\left( \omega ^{-2}+\frac{\omega ^{-2}}{2}\right) \right] +\sum_{i=2}^{n-1}\left[(m+1)-\omega ^{2}\omega ^{-2}\right] =nm-1.
\end{align*}
\end{example}

The significance of the above example lies in the fact that if $m=2$, then for every $n \in \Bbb{N}$, we can provide a fusion frame $W$ and a G\~{a}vru\c{t}a dual $V$ of $W$ such that $e(V)-e(W)=n$. However, this statement is not possible in the context of ordinary frames.

\bibliographystyle{amsplain}

\end{document}